\documentclass{amsart}
\usepackage[all]{xy}
\usepackage{enumerate}
\usepackage{mathrsfs}
\usepackage{amssymb}
\usepackage{graphicx}
\numberwithin{equation}{subsection}

\theoremstyle{plain}
\newtheorem{theorem}{Theorem}[section]
\newtheorem{lemma}[theorem]{Lemma}
\newtheorem{corollary}[theorem]{Corollary}

\theoremstyle{definition}

\theoremstyle{remark}
\newtheorem{remark}[theorem]{Remark}

\newcommand\Aut{\operatorname{Aut}}
\newcommand\Out{\operatorname{Out}}

\newcommand\Hom{\operatorname{Hom}}

\newcommand\GL{\operatorname{GL}}
\newcommand\SL{\operatorname{SL}}
\newcommand\IA{\operatorname{IA}}
\newcommand\IO{\operatorname{IO}}

\newcommand\id{\operatorname{id}}
\newcommand\pr{\operatorname{pr}}

\newcommand\im{\operatorname{im}}

\newcommand\Z{\mathbb{Z}}
\newcommand\Q{\mathbb{Q}}
\newcommand\Com{\mathcal{C}om}
\newcommand\calP{\mathcal{P}}
\newcommand\calC{\mathcal{C}}
\newcommand\calO{\mathcal{O}}
\newcommand\calG{\mathcal{G}}
\newcommand\wBr{\mathsf{wBr}}
\newcommand\dwBr{\mathsf{dwBr}}
\newcommand\gpS{\mathfrak{S}}
\newcommand\ab{\operatorname{ab}}

\newcommand\calA{\mathcal{A}}
\newcommand\calH{\mathcal{H}}

\newcommand\Prim{\operatorname{Prim}}

\title[Albanese homology of the IA-automorphism groups of free groups]{The stable Albanese homology of \\
the IA-automorphism groups of free groups}
\author{Mai Katada}
\address{Faculty of Mathematics, Kyushu University, Fukuoka 819-0395, Japan}
\curraddr{Graduate School of Mathematical Sciences, University of Tokyo, Tokyo 153-8914, Japan}
\email{mkatada@ms.u-tokyo.ac.jp}
\begin{document}

\begin{abstract}
The IA-automorphism group $\IA_n$ of the free group $F_n$ of rank $n$ is a normal subgroup of the automorphism group $\Aut(F_n)$ of $F_n$.
We study the Albanese homology of $\IA_n$, which is the quotient of the rational homology of $\IA_n$ defined as the image of the map induced by the abelianization map of $\IA_n$ on homology.
The Albanese homology of $\IA_n$ is an algebraic $\GL(n,\Q)$-representation.
We determine the representation structure of the Albanese homology of $\IA_n$ for $n$ greater than or equal to three times the homological degree. 
We also determine the structure of the stable Albanese homology of the analogue of $\IA_n$ to the outer automorphism group of $F_n$.
Moreover, we identify the relation between the stable Albanese (co)homology of $\IA_n$ and the stable cohomology of $\Aut(F_n)$ with certain twisted coefficients.
\end{abstract}
\maketitle
\section{Introduction}

The IA-automorphism group $\IA_n$ of the free group $F_n$ of rank $n$ is the normal subgroup of the automorphism group $\Aut(F_n)$ of $F_n$ that is trivial under the canonical group homomorphism from $\Aut(F_n)$ to the general linear group $\GL(n,\Z)$ induced by the abelianization map of $F_n$.
Then we have a short exact sequence of groups 
\begin{gather*}
    1\to \IA_n\to \Aut(F_n)\to \GL(n,\Z)\to 1.
\end{gather*}
By this short exact sequence, the (co)homology of $\IA_n$ admits an action of $\GL(n,\Z)$.
The IA-automorphism group $\IA_n$ is analogous to the \emph{Torelli groups} for surfaces, which are important objects in low-dimensional topology.
Some strategies of studying the (co)homology of the Torelli groups can be used to study the (co)homology of $\IA_n$ and vice versa.

The structure of the first (co)homology was determined by
 Bachmuth \cite{Bachmuth} as a free abelian group, and independently by Formanek \cite{Formanek},
Cohen--Pakianathan, Farb (both unpublished) and Kawazumi \cite{Kawazumi} as a $\GL(n,\Z)$-module.
The \emph{Johnson homomorphism} for $\Aut(F_n)$ induces an isomorphism
\begin{gather*}
    H_1(\IA_n,\Z)\xrightarrow{\cong} \Hom(H_{\Z},{\bigwedge}^2 H_{\Z}),\quad H_{\Z}=H_1(F_n,\Z).
\end{gather*}
For $n=3$, it is known that $\IA_3$ is not finitely presentable by Krsti\'{c}--McCool \cite{KM} and $H_2(\IA_3,\Z)$ has infinite rank by Bestvina--Bux--Margalit \cite{BBM}.
Pettet \cite{Pettet} determined the $\GL(n,\Z)$-subrepresentation of $H^2(\IA_n,\Q)$ that is detected by using the Johnson homomorphism, which is regarded as the second Albanese cohomology $H^2_A(\IA_n,\Q)$ of $\IA_n$ explained below.
Satoh \cite{Satoh2021} detected an irreducible
subrepresentation of $H^2(\IA_3,\Q)$ which is not included in $H_A^2(\IA_3,\Q)$.
For $n\ge 4$, it is still open whether $\IA_n$ is finitely presentable or not.
However, it is known that $H_2(\IA_n,\Z)$ is finitely generated as a $\GL(n,\Z)$-representation by Day--Putman \cite{DP}.

In a stable range, that is, for sufficiently large $n$ with respect to the (co)homological degree, the structure of the rational (co)homology of $\IA_n$ has been studied \cite{Katada, HK, LindellK}, and we have a conjectural structure of the stable rational (co)homology of $\IA_n$ (see Theorem \ref{HKtheorem}).

The main interest of this paper is a subalgebra of the rational cohomology of $\IA_n$ which seems to play an essential role in the stable rational cohomology of $\IA_n$.
The subalgebra is defined to be the image of the map induced by the abelianization map of $\IA_n$:
\begin{gather*}
    H_A^*(\IA_n,\Q)=\im (H^*(\IA_n^{\ab},\Q)\to H^*(\IA_n,\Q)).
\end{gather*}
Church--Ellenberg--Farb \cite{CEF} called $H_A^*(\IA_n,\Q)$ the \emph{Albanese cohomology} of $\IA_n$.
The \emph{Albanese homology} $H^A_*(\IA_n,\Q)$ is predual to the Albanese cohomology defined by
\begin{gather*}
     H^A_*(\IA_n,\Q)=\im (H_*(\IA_n,\Q)\to H_*(\IA_n^{\ab},\Q)).
\end{gather*}
It follows from the definition of the Albanese (co)homology of $\IA_n$ and the computation of the first homology of $\IA_n$ that the Albanese (co)homology is an \emph{algebraic $\GL(n,\Q)$-representation}.
The second and the third Albanese homology of $\IA_n$ was determined by Pettet \cite{Pettet} and the author \cite{Katada}, respectively.
Moreover, in \cite{Katada}, the author detected a large subquotient $\GL(n,\Q)$-representation $W_i$ of $H^A_i(\IA_n,\Q)$ for each $n\ge 3i$, and conjectured that $H^A_i(\IA_n,\Q)$ is isomorphic to $W_i$ for $n\ge 3i$.

The aim of this paper is to prove this conjecture on the representation structure of the Albanese homology of $\IA_n$.

\begin{theorem}[Theorem \ref{TheoremAlbaneseIAn}, cf.\;{\cite[Conjecture 6.2]{Katada}}]\label{intromaintheorem}
    We have an isomorphism of $\GL(n,\Q)$-representations
    \begin{gather*}
        F_i: H_i^A(\IA_n,\Q)\xrightarrow{\cong} W_i
    \end{gather*}
    for $n\ge 3i$.
\end{theorem}

It follows from Theorem \ref{intromaintheorem} that the Albanese homology of $\IA_n$ is representation stable in $n\ge 3i$ in the sense of Church--Farb \cite{CFrep}.

The author received a draft version of \cite{LindellK} by Erik Lindell and noticed that some reinterpretation of \cite[Proposition 6.3]{LindellK} can be used to determine the structure of the stable Albanese homology of $\IA_n$.
In the appendix of \cite{LindellK}, she proved the statement of Theorem \ref{intromaintheorem} only for $n\gg i$.

Habiro and the author \cite{HK} studied the structure of the stable rational cohomology of $\IA_n$.
By using Theorem \ref{intromaintheorem}, we can remove from one of the main results of \cite{HK} (cf. \cite[Theorem 1.10 and Remark 7.9]{HK}) the assumption about the structure of the Albanese homology of $\IA_n$, and obtain the following theorem.

\begin{theorem}[Cf.\;{\cite[Theorem 1.10 and Remark 7.9]{HK}}]\label{HKtheorem}
   Suppose that $H^i(\IA_n,\Q)$ is an algebraic $\GL(n,\Q)$-representation for $n\gg i$.
   Then for $n\gg i$, we have
   $$H^i(\IA_n,\Q)\cong \bigoplus_{k+l=i}W_{k}^*\otimes \Q[z_1,z_2,\cdots]_{l},$$
   where $\Q[z_1,z_2,\cdots]_{l}$ denotes the degree $l$ part of $\Q[z_1,z_2,\cdots]$ and $\deg z_j=4j$.
\end{theorem}

Lindell \cite{LindellK} has recently weakened the assumption of Theorem \ref{HKtheorem} that the family $\{H^*(\IA_n,\Q)\}_n$ is algebraic for $n\gg *$ to the assumption that $\{H^*(\IA_n,\Q)\}_n$ satisfies \emph{Borel vanishing} for $n\gg *$ (see \cite[Definition 1.5]{LindellK}).

In this paper, we will prove Theorem \ref{intromaintheorem} by taking care of the stable range. 
It follows from the proof of Theorem \ref{intromaintheorem} that the Albanese cohomology algebra is quadratic for $n\ge 3*$.
We also prove several related conjectures which are proposed in \cite{Katada}.
In particular, we will determine the structure of the Albanese homology of the analogue $\IO_n$ of $\IA_n$ to the outer automorphism group of $F_n$ for $n\ge 3*$.

We will also prove the conjecture on the relation between the stable Albanese cohomology of $\IA_n$ and the stable cohomology of $\Aut(F_n)$ with coefficients in the tensor product $H^{p,q}$ of $p$ copies of the standard representation $H=H_1(F_n,\Q)$ of $\GL(n,\Q)$ and $q$ copies of the dual representation $H^*$.

\begin{theorem}[Theorem \ref{TheoremAutIA}, cf.\;{\cite[Conjecture 7.2]{HK}}]
    The inclusion map $i: \IA_n\hookrightarrow \Aut(F_n)$ induces an isomorphism of $\Q[\gpS_{p}\times \gpS_{q}]$-modules
    \begin{gather*}
        i^*: H^*(\Aut(F_n),H^{p,q})\xrightarrow{\cong} [H_A^*(\IA_n,\Q)\otimes H^{p,q}]^{\GL(n,\Z)}
    \end{gather*}
    for $n\ge \max(3*,\, 2\!\ast\!+p+q+3)$.
\end{theorem}

\section*{Acknowledgements}
The author would like to thank Erik Lindell for sending a draft version of \cite{LindellK} to her and for helpful comments.
She also thanks Kazuo Habiro for valuable discussions.
She was supported in part by JSPS KAKENHI Grant Number JP22KJ1864 and JP24K16916.

\section{The stable Albanese homology of $\IA_n$}

In this section, we prove \cite[Conjecture 6.2]{Katada} on the $\GL(n,\Q)$-representation structure of the stable Albanese homology of $\IA_n$.

\subsection{Algebraic $\GL(n,\Q)$-representations}
Here we briefly recall some facts from representation theory of $\GL(n,\Q)$. See Fulton--Harris \cite{FH} for details.

A finite-dimensional $\GL(n,\Q)$-representation $(\rho,V)$ is called \emph{algebraic} if after choosing a basis for $V$, the $(\dim V)^2$ coordinate functions of the group homomorphism $\rho:\GL(n,\Q)\to \GL(V)$ are rational functions on $n^2$ variables.

It is well known that algebraic $\GL(n,\Q)$-representations are completely reducible and that irreducible representations are classified by \emph{bipartitions}, i.e., pairs of partitions.
Here, a partition $\lambda=(\lambda_1,\lambda_2,\cdots,\lambda_l)$ is a non-increasing sequence of non-negative integers.
Let $|\lambda|=\sum_{i=1}^{n} \lambda_i$ denote the \emph{size} of $\lambda$ and $l(\lambda)=\max(\{0\}\cup \{i\mid \lambda_i>0\})$ the \emph{length} of $\lambda$.
For a bipartition $(\lambda,\mu)$, the irreducible algebraic $\GL(n,\Q)$-representation $V_{\lambda,\mu}$ is constructed as follows.
Consider $H=H_1(F_n,\Q)$ as the standard representation of $\GL(n,\Q)$ and set $H^{p,q}=H^{\otimes p}\otimes (H^*)^{\otimes q}$ for $p,q\ge 0$.
The \emph{traceless part} $T_{p,q}$ of $H^{p,q}$ is defined by
\begin{gather*}
    T_{p,q}=\bigcap_{1\le k\le p,\; 1\le l\le q} \ker (c_{k,l}: H^{p,q}\to H^{p-1,q-1}),
\end{gather*}
where $c_{k,l}$ is the \emph{contraction map} that takes the dual pairing $\langle,\rangle: H\otimes H^*\to \Q$ at the $k$-th tensorand of $H^{\otimes p}$ and the $l$-th tensorand of $(H^*)^{\otimes q}$.
For $|\lambda|=p$, $|\mu|=q$, let 
\begin{gather*}
    V_{\lambda,\mu}=T_{p,q}\otimes_{\Q[\gpS_p\times \gpS_q]} (S^{\lambda}\otimes S^{\mu}),
\end{gather*}
where $S^{\lambda}$ and $S^{\mu}$ denote the Specht modules corresponding to $\lambda$ and $\mu$, respectively.
If $n\ge l(\lambda)+l(\mu)$, then $V_{\lambda,\mu}$ is an irreducible algebraic $\GL(n,\Q)$-representation and otherwise $V_{\lambda,\mu}=0$.

\subsection{Invariant theory of $\GL(n,\Z)$}

Let $\{e_i\mid 1\le i\le n\}$ be a basis for $H$ and $\{e_i^*\mid 1\le i\le n\}$ the dual basis for $H^*$.
Define a linear map $\omega:\Q\to H\otimes H^*$ by 
\begin{gather}\label{omega}
    \omega(1)=\sum_{i=1}^{n} e_i\otimes e_i^*,
\end{gather}
which is dual to the pairing $\langle,\rangle$.
By using $(p+q)$ copies of the element $\omega(1)$, we obtain a linear map
\begin{gather*}
   \Omega: \Q[\gpS_{p+q}]\to [H^{p,q}\otimes H^{q,p}]^{\GL(n,\Z)},
\end{gather*}
which is defined by 
\begin{gather*}
 \Omega(\sigma)=\sum_{1\le i_1,\cdots,i_{p+q}\le n} \left(\bigotimes_{j=1}^{p} e_{i_j}\otimes \bigotimes_{j=p+1}^{p+q} e_{i_{\sigma^{-1}(j)}}^*\right)\otimes \left(\bigotimes_{j=p+1}^{p+q}e_{i_{j}}\otimes \bigotimes_{j=1}^{p} e_{i_{\sigma^{-1}(j)}}^*\right)
\end{gather*}
for $\sigma\in \gpS_{p+q}$.
It follows from the first fundamental theorem of invariant theory for $\SL$ that the map $\Omega$ is surjective (see \cite[Section 2.1]{LindellK} for details).

Let $\pr: [H^{p,q}\otimes H^{q,p}]^{\GL(n,\Z)}\twoheadrightarrow [T_{p,q}\otimes T_{q,p}]^{\GL(n,\Z)}$ denote the projection.
Let
\begin{gather*}
    \Omega'=\pr\circ \Omega|_{\Q[\gpS_p\times \gpS_q]}: \Q[\gpS_p\times \gpS_q]\to[T_{p,q}\otimes T_{q,p}]^{\GL(n,\Z)},
\end{gather*}
where $\Omega|_{\Q[\gpS_p\times \gpS_q]}$ is the restriction of $\Omega$ to $\Q[\gpS_p\times \gpS_q]$.
By the surjectivity of $\Omega$, the map $\Omega'$ is also surjective
since we have $\pr\circ \Omega(\sigma)=0$ for $\sigma\in \gpS_{p+q}\setminus (\gpS_{p}\times \gpS_{q})$.
Let $\Q[\gpS_p\times \gpS_q]$ act on $[T_{p,q}\otimes T_{q,p}]^{\GL(n,\Z)}$ by the place permutations of $p$ copies of $H^*$ and $q$ copies of $H$ in $T_{q,p}$.
Then $\Omega'$ is a $\Q[\gpS_p\times \gpS_q]$-module map.

We generalize \cite[Proposition 2.10]{LindellK} to obtain the following lemma.

\begin{lemma}\label{Propimprovedinvariant}
For $n\ge \max(p+q,r+s)$, we have a linear isomorphism 
    \begin{gather*}
        [T_{p,q}\otimes T_{r,s}]^{\GL(n,\Z)}\cong 
        \begin{cases}
            \Q[\gpS_p\times \gpS_q] & (p=s,q=r)\\
            0 & (\text{otherwise}).
        \end{cases}
    \end{gather*}
Therefore, the map $\Omega'$ is an isomorphism of $\Q[\gpS_p\times \gpS_q]$-modules for $n\ge p+q$.
\end{lemma}
\begin{proof}
We first compute the dimension of $[T_{p,q}\otimes T_{r,s}]^{\GL(n,\Z)}$.
By Koike \cite{Koike}, for $n\ge \max(p+q,r+s)$, we have irreducible decompositions
    \begin{gather*}
        T_{p,q}\cong \bigoplus_{\lambda\vdash p, \mu\vdash q} (V_{\lambda,\mu})^{\oplus (\dim S^{\lambda}\dim S^{\mu})},\quad T_{r,s}\cong \bigoplus_{\xi\vdash r, \eta\vdash s} (V_{\xi,\eta})^{\oplus (\dim S^{\xi}\dim S^{\eta})}.
    \end{gather*}
    Therefore, we have
    \begin{gather*}
        T_{p,q}\otimes T_{r,s}\cong \bigoplus_{\substack{\lambda\vdash p,\mu\vdash q\\ \xi\vdash r,\eta\vdash s}} 
        (V_{\lambda,\mu}\otimes V_{\xi,\eta})^{\oplus (\dim S^{\lambda}\dim S^{\mu}\dim S^{\xi}\dim S^{\eta})}.
    \end{gather*}
    Since for each $\lambda\vdash p$, $\mu\vdash q$, $\xi\vdash r$, $\eta\vdash s$,
    \begin{gather*}
        [V_{\lambda,\mu}\otimes V_{\xi,\eta}]^{\GL(n,\Z)}\cong \Hom_{\GL(n,\Z)}(V_{\eta,\xi},V_{\lambda,\mu})\cong \Q^{\oplus (\delta_{\lambda,\eta}\delta_{\mu,\xi})}
    \end{gather*}
    by Schur's lemma,
    we have
    \begin{gather*}
        \dim([T_{p,q}\otimes T_{r,s}]^{\GL(n,\Z)})=\delta_{p,s}\delta_{q,r}\sum_{\lambda\vdash p,\mu\vdash q} (\dim S^{\lambda})^2 (\dim S^{\mu})^2.
    \end{gather*}
    
    It follows from 
     $$\Q[\gpS_p\times \gpS_q]\cong  \Q[\gpS_p]\otimes \Q[\gpS_q]\cong\bigoplus_{\lambda\vdash p,\mu\vdash q}(S^{\lambda})^{\oplus \dim S^\lambda}\otimes (S^{\mu})^{\oplus \dim S^\mu}$$
     that we have
    \begin{gather*}
        \dim(\Q[\gpS_p\times \gpS_q])=\sum_{\lambda\vdash p,\mu\vdash q} (\dim S^{\lambda})^2 (\dim S^{\mu})^2,
    \end{gather*}
which completes the proof.
\end{proof}

\subsection{The conjectural structure of the stable Albanese homology of $\IA_n$}
We will briefly review our conjectural structure of the stable Albanese homology of $\IA_n$.
For $i\ge 1$, let 
$$U_i=\Hom(H,{\bigwedge}^{i+1}H).$$
Let $U_*=\bigoplus_{i\ge 1}U_i$ be the graded $\GL(n,\Q)$-representation.
Define $W_*=\widetilde{S}^*(U_*)$ as the \emph{traceless part} of the graded-symmetric algebra $S^*(U_*)$ of $U_*$.
Here, the \emph{traceless tensor product} $V_{\lambda,\mu}\widetilde{\otimes}\: V_{\xi,\eta}$ of two irreducible algebraic $\GL(n,\Q)$-representations $V_{\lambda,\mu}$ and $V_{\xi,\eta}$ is defined by
\begin{gather*}
    V_{\lambda,\mu}\widetilde{\otimes}\: V_{\xi,\eta}=(V_{\lambda,\mu}\otimes V_{\xi,\eta})\cap T_{|\lambda|+|\xi|,|\mu|+|\eta|}\subset H^{|\lambda|+|\xi|,|\mu|+|\eta|},
\end{gather*}
and the traceless part of the tensor algebra is defined by using the traceless tensor product instead of the usual tensor product.
The traceless part of the graded-symmetric algebra is defined as the image of the traceless part of the tensor algebra under the canonical projection.
See \cite[Sections 2.5 and 2.6]{Katada} for details of the notion of the traceless part.

In order to prove our main theorem, we will review and give the stable range of \cite[Proposition 12.3 and Lemma 12.4]{Katada}. To state these proposition and lemma, we will review the wheeled PROP $\calC_{\calP_0^{\circlearrowright}}$ that is introduced in \cite{Kawazumi--Vespa}, and the non-unital wheeled PROP $\calC_{\calO^{\circlearrowright}}$ that is introduced in \cite{Katada}, corresponding to the operad $\Com$ of non-unital commutative algebras. 
Please refer to \cite[Section 12]{Katada} for more details.

Let $\calP_0=\bigoplus_{k\ge 1}\calP_0(k)$ denote the operadic suspension of the operad $\Com$, i.e., we have $\calP_0(0)=0$ and $\calP_0(k)$ is the sign representation of $\gpS_k$ placed in cohomological dimension $k-1$ for $k\ge 1$.
Let $\calP_0^{\circlearrowright}$ denote the wheeled completion of $\calP_0$ and $\calC_{\calP_0^{\circlearrowright}}$ the wheeled PROP freely generated by $\calP_0^{\circlearrowright}$.

Let $\calO=\bigoplus_{k\ge 2}\calP_0(k)$ denote the non-unital suboperad of $\calP_0$. Let $\calO^{\circlearrowright}$ denote the non-unital wheeled sub-operad of $\calP_0^{\circlearrowright}$, which consists of the non-unital operad $\calO$, a right $\calO$-module $\calO^{\circlearrowright}_w$ and the contraction maps $\xi_{i}:\calO(n)\to \calO^{\circlearrowright}_w(n-1), 1\le i\le n$.
Let $\calC_{\calO^{\circlearrowright}}$ denote the non-unital wheeled sub-PROP of $\calC_{\calP_0^{\circlearrowright}}$.

\begin{remark}\label{remark2.2}
In \cite{LindellK}, Lindell defined a $\wBr_n$-module $\calP\otimes \det$ (resp. a $\dwBr$-module $\calP'\otimes \det$), which is a functor from the walled Brauer category $\wBr_n$ (resp. the downward walled Brauer category $\dwBr$) to the category of $\Q$-vector spaces,
in terms of labelled partitions of sets.
A wheeled PROP (resp. a non-unital wheeled PROP) admits a natural structure of a $\wBr_n$-module (resp. a $\dwBr$-module), and 
the wheeled PROP $\calC_{\calP_0^{\circlearrowright}}$ corresponds to $\calP\otimes \det$ and the non-unital wheeled PROP $\calC_{\calO^{\circlearrowright}}$ corresponds to $\calP'\otimes \det$.
\end{remark}

\begin{lemma}[Cf.\;{\cite[Proposition 12.3]{Katada}}]\label{lemmaW*}
    For $n\ge 3i$, we have an isomorphism of $\GL(n,\Q)$-representations
    \begin{gather*}
        W_i\cong \bigoplus_{a-b=i}T_{a,b}\otimes_{\Q[\gpS_a\times \gpS_b]} \calC_{\calO^{\circlearrowright}}(a,b).
    \end{gather*}
    Here, we have $T_{a,b}\otimes_{\Q[\gpS_a\times \gpS_b]} \calC_{\calO^{\circlearrowright}}(a,b)=0$ unless $i\le a\le 2i$, $0\le b\le i$.
\end{lemma}
\begin{proof}
    The proof of \cite[Proposition 12.3]{Katada} does not take care of the stable range, but the argument holds for $n\ge 3i$ 
    since for any $k,l\ge 1$ such that $kl\le i$, we have
    \begin{gather*}
        T_{k(l+1),k}\otimes_{\Q[\gpS_{l+1} \wr \gpS_{k}]}\calO(l+1)^{\otimes k}\cong \widetilde{S}^{k}(V_{1^{l+1},1}),\: 
        T_{kl, 0} \otimes_{\Q[\gpS_{l}\wr \gpS_{k}]} \calO_w^{\circlearrowright}(l)^{\otimes k}\cong \widetilde{S}^{k}(V_{1^l,0}),
    \end{gather*}
     where $\widetilde{S}^k(V_{1^l,i})$ denotes the traceless part of the graded-symmetric power of $V_{1^l,i}$ ($i=0,1$).
\end{proof}

\begin{lemma}[Cf.\;{\cite[Lemma 12.4]{Katada}}]\label{lemmaW*inv}
    For $n\ge \max(3i,p+q)$, we have an isomorphism of $\Q[\gpS_p\times \gpS_q]$-modules
    \begin{gather*}
        [(W_i)^*\otimes H^{p,q}]^{\GL(n,\Z)}\cong \calC_{\calP_0^{\circlearrowright}}(p,q)_i.
    \end{gather*}
\end{lemma}

\begin{proof}
Since the proof of \cite[Lemma 12.4]{Katada} does not deal with the stable range, we will recall the proof of it to take care of the stable range.
By Lemma \ref{lemmaW*}, we have for $n\ge 3i$,
\begin{gather*}
        (W_i)^*\cong \bigoplus_{a-b=i}T_{b,a}\otimes_{\Q[\gpS_a\times \gpS_b]} \calC_{\calO^{\circlearrowright}}(a,b),
    \end{gather*}
where the direct summand is trivial unless $i\le a\le 2i, 0\le b\le i$.
For $n\ge p+q$, we have
$H^{p,q}\cong \bigoplus_{c=0}^{\min(p,q)}T_{p-c,q-c}^{\oplus \binom{p}{c}\binom{q}{c}c!}$, which can be checked by considering the graded quotient of the filtration of $H^{p,q}$ that is defined by the kernel of the composition of $c$-tuples of contraction maps $c_{k,l}$. (See \cite[Lemma 4.8]{HKBorel} for example.)
Therefore, by Lemma \ref{Propimprovedinvariant}, we have
\begin{gather*}
\begin{split}
&[(W_i)^*\otimes H^{p,q}]^{\GL(n,\Z)}\\
&\cong \left[\bigoplus_{a-b=i}\left(T_{b,a}\otimes_{\Q[\gpS_a\times \gpS_b]} \calC_{\calO^{\circlearrowright}}(a,b)\right)\otimes \bigoplus_{c=0}^{\min(p,q)}\left(T_{p-c,q-c}^{\oplus \binom{p}{c}\binom{q}{c}c!}\right)\right]^{\GL(n,\Z)}\\
&\cong \bigoplus_{a-b=i}\bigoplus_{c=0}^{\min(p,q)}[T_{p-c,q-c}\otimes T_{b,a}]^{\GL(n,\Z)}\otimes_{\Q[\gpS_a\times \gpS_b]} \calC_{\calO^{\circlearrowright}}(a,b)^{\oplus \binom{p}{c}\binom{q}{c}c!}\\
&\cong \bigoplus_{c=0}^{\min(p,q)}\Q[\gpS_{p-c}\times \gpS_{q-c}]\otimes_{\Q[\gpS_{p-c}\times \gpS_{q-c}]} \calC_{\calO^{\circlearrowright}}(p-c,q-c)_i^{\oplus \binom{p}{c}\binom{q}{c}c!}
\\
&\cong \bigoplus_{c=0}^{\min(p,q)}\calC_{\calO^{\circlearrowright}}(p-c,q-c)_i^{\oplus \binom{p}{c}\binom{q}{c}c!}.
\end{split}
\end{gather*}
Since the morphisms of $\calC_{\calP_0^{\circlearrowright}}$ are generated by those of $\calC_{\calO^{\circlearrowright}}$ and the identity morphism, we obtain
$$\calC_{\calP_0^{\circlearrowright}}(p,q)_i\cong \bigoplus_{c=0}^{\min(p,q)}\calC_{\calO^{\circlearrowright}}(p-c,q-c)_i^{\oplus \binom{p}{c}\binom{q}{c}c!},$$
which completes the proof.
\end{proof}

\subsection{The stable Albanese homology of $\IA_n$}

Now we state the main theorem of this paper.

\begin{theorem}[Cf.\;{\cite[Conjecture 6.2]{Katada}}]\label{TheoremAlbaneseIAn}
    For $n\ge 3i$, we have an isomorphism of $\GL(n,\Q)$-representations
    \begin{gather*}
        F_i: H_i^A(\IA_n,\Q)\to W_i.
    \end{gather*}
\end{theorem}

\begin{remark}
The range $n\ge 3i$ in the statement of Theorem \ref{TheoremAlbaneseIAn} might be improved. 
Theorem \ref{TheoremAlbaneseIAn} implies that the Albanese homology of $\IA_n$ is representation stable in $n\ge 3i$ in the sense of Church--Farb \cite{CFrep} since the multiplicity of an irreducible algebraic $\GL(n,\Q)$-representation which appears as a component of $W_i$ is stable for $n\ge 3i$
(see \cite[Section 6.1]{Katada} for details).
\end{remark}

Since the first Albanese homology of $\IA_n$ is isomorphic to the first rational homology of $\IA_n$, the degree $1$ case follows from \cite{Bachmuth, Formanek, Kawazumi}.
The degree $2$ case is proven by Pettet \cite{Pettet} and the degree $3$ case is by the author \cite{Katada}.
Moreover, in \cite[Theorem 6.1]{Katada}, the author proved the following.

\begin{theorem}[Cf.\;{\cite[Theorem 6.1]{Katada}}]\label{TheoremAlbaneseinclusion}
    We have a morphism of graded $\GL(n,\Q)$-representations
    \begin{gather*}
        F_*: H_*(U_1,\Q)\to S^*(U_*)
    \end{gather*}
    such that $F_*(H^A_*(\IA_n,\Q))\supset W_*$ for $n\ge 3*$.
\end{theorem}

\begin{remark}
In order to detect $W_*$ in the image of $F_*$, the author used abelian cycles in $H_*(\IA_n,\Q)$. Therefore, it follows from Theorem \ref{TheoremAlbaneseIAn} that the stable Albanese homology of $\IA_n$ is generated by abelian cycles.
\end{remark}

In the appendix of \cite{LindellK}, the author proved the statement of Theorem \ref{TheoremAlbaneseIAn} for sufficiently large $n$ with respect to the homological degree.

\begin{proof}[Proof of Theorem \ref{TheoremAlbaneseIAn}]
Let $\calG(p,q)$ denote the graded quotient vector space of $(2,1)$-valent marked directed oriented graphs with $p$ incoming legs and $q$ outgoing legs modulo the directed IH-relation, which is introduced in \cite[Section 6]{LindellK}.
In the proof of \cite[Proposition 6.3]{LindellK}, Lindell proved that $\calG(p,q)$ is isomorphic to $(\calP\otimes\det)(p,q)$, which is isomorphic to $\calC_{\calP_0^{\circlearrowright}}(p,q)$ (see Remark \ref{remark2.2}).
Moreover, he constructed a linear map
\begin{gather*}
    \alpha: \calG(p,q)\to
    [R_{\text{pres}}\otimes H^{p,q}]^{\GL(n,\Z)}
\end{gather*}
by using the element $\omega(1)$ in \eqref{omega}, where $R_{\text{pres}}$ is the graded-commutative ring that is constructed in \cite[Definition 6.1]{LindellK}.
By \cite[Remark 6.2]{LindellK}, we have an isomorphism of $\GL(n,\Q)$-representations
\begin{gather*}
    R_{\text{pres}}\cong H^*(U_1,\Q)/{\langle R_2 \rangle},
\end{gather*}
where $\langle R_2 \rangle$ denotes the two-sided ideal of $H^*(U_1,\Q)$ generated by 
$$R_2=\ker\left(H^2(U_1,\Q)\cong {\bigwedge}^2 H^1(\IA_n,\Q) \xrightarrow{\cup} H^2(\IA_n,\Q)\right).$$
Since the Albanese cohomology of $\IA_n$ is equal to the image of the cup product map $\bigwedge^* H^1(\IA_n,\Q) \xrightarrow{\cup} H^*(\IA_n,\Q)$,
we have a surjective morphism of graded $\GL(n,\Q)$-representations
\begin{gather}\label{surjhom}
    H^*(U_1,\Q)/{\langle R_2 \rangle}\twoheadrightarrow H^*_A(\IA_n,\Q).
\end{gather}
Therefore, we obtain the following composition of linear maps
\begin{gather*}
   \calC_{\calP_0^{\circlearrowright}}(p,q)\cong \calG(p,q)\xrightarrow{\alpha} [R_{\text{pres}}\otimes H^{p,q}]^{\GL(n,\Z)}\twoheadrightarrow  [H^*_A(\IA_n,\Q)\otimes H^{p,q}]^{\GL(n,\Z)}.
\end{gather*}

For $n\ge 3i$, we have $\calC_{\calP_0^{\circlearrowright}}(2i,i)_i\cong [(W_i)^{*}\otimes H^{2i,i}]^{\GL(n,\Z)}$ by Lemma \ref{lemmaW*inv}, and
the map $\alpha:\calG(2i,i)\to [R_{\text{pres}}\otimes H^{2i,i}]^{\GL(n,\Z)}$ is surjective.
Hence, we obtain a surjective map
\begin{gather*}
     [(W_i)^{*}\otimes H^{2i,i}]^{\GL(n,\Z)}\twoheadrightarrow [H^i_A(\IA_n,\Q)\otimes H^{2i,i}]^{\GL(n,\Z)}.
\end{gather*}
It follows from Theorem \ref{TheoremAlbaneseinclusion} that 
\begin{gather*}
    \dim_{\Q} [(W_i)^{*}\otimes H^{2i,i}]^{\GL(n,\Z)}\le \dim_{\Q} [H^i_A(\IA_n,\Q)\otimes H^{2i,i}]^{\GL(n,\Z)}
\end{gather*}
and thus we have 
\begin{gather*}
    [(W_i)^{*}\otimes H^{2i,i}]^{\GL(n,\Z)}\cong [H^i_A(\IA_n,\Q)\otimes H^{2i,i}]^{\GL(n,\Z)}.
\end{gather*}
Therefore, since
$H_i^A(\IA_n,\Q)\subset H_i(U_1,\Q)\cong \bigwedge^i U_1\subset H^{2i,i},$
we have 
$H^i_A(\IA_n,\Q)\cong (W_i)^{*},$
which completes the proof.
\end{proof}

\section{Several related conjectures}

We prove several conjectures related to the $\GL(n,\Q)$-representation structure of the stable Albanese homology of $\IA_n$, which are proposed in \cite{Katada}.

\subsection{The algebra structure of the Albanese cohomology of $\IA_n$}

Here we study the algebra structure of the Albanese cohomology of $\IA_n$.
The cup product gives the algebra structure on the rational cohomology of $\IA_n$ and $\IA_n^{\ab}$, and the Albanese cohomology of $\IA_n$ is a subalgebra of $H^*(\IA_n,\Q)$ generated by $H^1(\IA_n,\Q)$.
We refer the reader to \cite[Section 8]{Katada} for details.

The following theorem directly follows from the proof of Theorem \ref{TheoremAlbaneseIAn}.

\begin{theorem}[Cf.\;{\cite[Conjecture 8.2]{Katada}}]
The Albanese cohomology algebra is stably quadratic in $n\ge 3*$.
That is,
the surjective $\GL(n,\Q)$-equivariant morphism \eqref{surjhom} of graded algebras
    \begin{gather*}
        H^*(U_1,\Q)/{\langle R_2\rangle}\twoheadrightarrow H_A^*(\IA_n,\Q)
    \end{gather*}
    is an isomorphism of $\GL(n,\Q)$-representations for $n\ge 3*$.
\end{theorem}

\subsection{The coalgebra structure of the Albanese homology of $\IA_n$}

Here we study the coalgebra structure of the Albanese homology of $\IA_n$.
We refer the reader to \cite[Section 7]{Katada} for details.

The rational homology of $\IA_n$ (resp. $\IA_n^{\ab}$) has a natural coalgebra structure, which is predual to the algebra structure of $H^*(\IA_n,\Q)$ (resp. $H^*(\IA_n^{\ab},\Q)$). 
The coalgebra structures of $H_*(\IA_n,\Q)$ and $H_*(\IA_n^{\ab},\Q)$ induce a coalgebra structure on the Albanese homology of $\IA_n$.
We also have a coalgebra structure on the graded-symmetric algebra $S^*(U_*)$ (see \cite[Section 2.6]{Katada} for details).

Let $F_*:H_*(U_1,\Q)\to S^*(U_*)$ be the morphism that appeared in Theorem \ref{TheoremAlbaneseinclusion}.
By \cite[Proposition 7.1]{Katada}, the map $F_*$ is a coalgebra morphism.
Since we have
$$\Prim(S^*(U_*))=U_*=\Prim(W_*),$$
where $\Prim(C)$ denotes the primitive part of a co-augmented coalgebra $C$,
Theorem \ref{TheoremAlbaneseIAn} implies the following.

\begin{corollary}[Cf.\;{\cite[Conjecture 7.2]{Katada}}]
The $\GL(n,\Q)$-equivariant coalgebra morphism $F_*$ restricts to a $\GL(n,\Q)$-equivariant morphism
\begin{gather*}
    F_*:\Prim(H^A_*(\IA_n,\Q))\to U_*,
\end{gather*}
which is an isomorphism for $n\ge 3*$.
\end{corollary}

\subsection{The stable Albanese homology of $\IO_n$}
Let $\Out(F_n)$ denote the outer automorphism group of $F_n$.
Define $\IO_n$ as the kernel of the surjective group homomorphism from $\Out(F_n)$ to $\GL(n,\Z)$, that is, we have a short exact sequence of groups
\begin{gather*}
    1\to \IO_n\to \Out(F_n)\to \GL(n,\Z)\to 1.
\end{gather*}
The Albanese homology of $\IO_n$ is defined in a way similar to $\IA_n$ by
\begin{gather*}
 H^A_*(\IO_n,\Q)=\im (H_*(\IO_n,\Q)\to H_*(\IO_n^{\ab},\Q)).
\end{gather*}
The conjectural structure $W^O_*$ of the stable Albanese homology of $\IO_n$ is constructed as the traceless part of the graded-symmetric algebra $S^*(U^O_*)$ of the graded $\GL(n,\Q)$-representation 
$$U^O_*=\bigoplus_{i\ge 1}U^O_i, \quad
\begin{cases}
    U^O_1=\Hom(H,\bigwedge^2 H)/H\cong V_{1^2,1} & i=1\\
    U^O_i=U_i & i\ge 2.
\end{cases}$$
We refer the reader to \cite[Section 9]{Katada} for details.

By using Theorem \ref{TheoremAlbaneseIAn} and several results in \cite{Katada}, we determine the stable Albanese homology of $\IO_n$.
The cases of degree $1,2$ and $3$ are proven by \cite{Kawazumi}, \cite{Pettet} and \cite{Katada}, respectively.

\begin{theorem}[Cf.\;{\cite[Conjecture 9.7]{Katada}}]\label{TheoremAlbaneseIO}
    We have an isomorphism of $\GL(n,\Q)$-representations
    \begin{gather*}
        H^A_i(\IO_n,\Q)\cong W^O_i
    \end{gather*}
    for $n\ge 3i$.
\end{theorem}

\begin{proof}
    If we disregard the stable range, then the statement follows directly from Theorem \ref{TheoremAlbaneseIAn} and \cite[Proposition 9.11]{Katada}. However, we will recall the argument and take care of the stable range.

    We have isomorphisms of $\GL(n,\Q)$-representations
    \begin{gather*}
        H^A_i(\IA_n,\Q)\cong H^A_i(\IO_n,\Q)\oplus (H^A_{i-1}(\IO_n,\Q)\otimes H)
    \end{gather*}
    for $n\ge 2$ by \cite[Proposition 9.8]{Katada}, and 
    \begin{gather*}
        W_i\cong W_i^{O}\oplus (W_{i-1}^O\otimes H)
    \end{gather*}
    for $n\ge 3i$ by \cite[Lemma 9.9]{Katada}.
    Since we have $H^A_1(\IO_n,\Q)\cong W_1^O$, the statement follows by induction on $i$.
\end{proof}

In the case of $\IA_n$, the polynomiality of the dimension of the stable Albanese homology was known in \cite{CEF}. By Theorem \ref{TheoremAlbaneseIAn} and an irreducible decomposition of $W_*$, it is possible to compute the exact polynomial.
It seems natural to expect the same thing holds for $\IO_n$, but the author has not found any literature about the polynomiality of the stable Albanese homology of $\IO_n$. By Theorem \ref{TheoremAlbaneseIO}, we obtain the polynomiality in the case of $\IO_n$ as well.

\begin{corollary}[Cf.\;{\cite[Conjecture 9.6]{Katada}}]
    There is a polynomial $P^O_i(T)$ of degree $3i$ such that we have 
    $\dim_{\Q}(H^A_i(\IO_n,\Q))=P^O_i(n)$ for $n\ge 3i$.
\end{corollary}

\subsection{Relation between the stable Albanese cohomology of $\IA_n$ and the stable twisted cohomology of $\Aut(F_n)$}

Here we study the relation between the stable Albanese cohomology of $\IA_n$ and the stable cohomology of $\Aut(F_n)$ with coefficients in $H^{p,q}$.

\begin{theorem}[Cf.\;{\cite[Conjectures 12.5 and 12.6]{Katada}}]\label{AutIA}
    We have an isomorphism of $\Q[\gpS_p\times \gpS_q]$-modules
    \begin{gather*}
        H^i(\Aut(F_n),H^{p,q})\cong [H_A^i(\IA_n,\Q)\otimes H^{p,q}]^{\GL(n,\Z)}
    \end{gather*}
    for 
    $n\ge \max(3i, 2i+p+q+3)$.
    The statement also holds for coefficients in $V_{\lambda,\mu}$ with $|\lambda|=p, |\mu|=q$ instead of $H^{p,q}$.
\end{theorem}

\begin{proof}
By \cite[Theorem A]{Lindell}, for $n\ge 2i+p+q+3$, we have isomorphisms of $\Q[\gpS_p\times \gpS_q]$-modules
    \begin{gather}\label{isomlindell}
     H^i(\Aut(F_n),H^{p,q})\cong (\calP\otimes \det)(p,q)_i\cong \calC_{\calP_0^{\circlearrowright}}(p,q)_i.
    \end{gather}
On the other hand, by Theorem \ref{TheoremAlbaneseIAn} and Lemma \ref{lemmaW*inv}, we have isomorphisms of $\Q[\gpS_p\times \gpS_q]$-modules
\begin{gather*}
    [H_A^i(\IA_n,\Q)\otimes H^{p,q}]^{\GL(n,\Q)}\cong [(W_i)^*\otimes H^{p,q}]^{\GL(n,\Q)}\cong \calC_{\calP_0^{\circlearrowright}}(p,q)_i
\end{gather*}
for $n\ge \max(3i,p+q)$, which completes the proof.
\end{proof}

Moreover, in \cite{HK}, Habiro and the author conjectured that the isomorphism in Theorem \ref{AutIA} is realized by the morphism of $\Q[\gpS_p\times \gpS_q]$-modules
\begin{gather*}
  i^*:  H^*(\Aut(F_n),H^{p,q}) \to H^*(\IA_n,H^{p,q})^{\GL(n,\Z)}=[H^*(\IA_n,\Q)\otimes H^{p,q}]^{\GL(n,\Z)}
\end{gather*}
that is induced by the inclusion map $i:\IA_n\hookrightarrow \Aut(F_n)$.
To prove this conjecture, we construct a wheeled PROP corresponding to $[H_A^*(\IA_n,\Q)\otimes H^{p,q}]^{\GL(n,\Z)}$ in a way similar to the wheeled PROP $\calH$ corresponding to $H^*(\Aut(F_n),H^{p,q})$ that was introduced in \cite{Kawazumi--Vespa}.

In what follows, let $H(n)$ denote $H=H_1(F_n,\Q)$ to make the dependence on $n$ explicit.
The inclusion map $F_n\hookrightarrow F_{n+1}$ that maps the generator $x_i\in F_n$ to $x_i\in F_{n+1}$ for $1\le i\le n$ induces an inclusion map $\IA_n\hookrightarrow \IA_{n+1}$.
We have the projection $H(n+1)\twoheadrightarrow H(n)$
that sends the basis element $e_i\in H(n+1)$ to $e_i\in H(n)$ for $1\le i\le n$ and $e_{n+1}\in H(n+1)$ to $0$.
We also have the projection
$H(n+1)^*\twoheadrightarrow H(n)^*$
that sends the dual basis element $e_i^*\in H(n+1)^*$ to $e_i^*\in H(n)^*$ and $e_{n+1}^*\in H(n+1)^*$ to $0$.
These three maps induce the linear map
\begin{gather*}
    [H^*_A(\IA_{n+1},\Q)\otimes H^{p,q}(n+1)]^{\GL(n+1,\Z)}\to [H^*_A(\IA_n,\Q)\otimes H^{p,q}(n)]^{\GL(n,\Z)}
\end{gather*}
and the stable $\GL(n,\Z)$-invariant part of the twisted Albanese cohomology is the limit $\varprojlim_{n} [H^*_A(\IA_n,\Q)\otimes H^{p,q}(n)]^{\GL(n,\Z)}$.
By Theorem \ref{TheoremAlbaneseIAn}, we have
\begin{gather*}
    \varprojlim_{n} [H^i_A(\IA_n,\Q)\otimes H^{p,q}(n)]^{\GL(n,\Z)}\cong [H^i_A(\IA_n,\Q)\otimes H^{p,q}(n)]^{\GL(n,\Z)}
\end{gather*}
for $n\ge \max(3i,p+q)$.

Define a wheeled PROP $\calA$ as follows.
The hom-spaces are the graded $\Q[\gpS_p\times \gpS_q]$-modules
\begin{gather*}
   \calA(p,q)=\varprojlim_{n} [H^*_A(\IA_n,\Q)\otimes H^{p,q}(n)]^{\GL(n,\Z)},
\end{gather*}
where the actions of $\gpS_p$ and $\gpS_q$ are given by place permutations of the copies of $H$ and $H^*$.
The horizontal composition
\begin{gather*}
    \calA(p_1,q_1)\otimes \calA(p_2,q_2)\to \calA(p_1+p_2,q_1+q_2)
\end{gather*}
is induced by the cup product map for the cohomology of $\IA_n$, and the vertical composition
\begin{gather*}
    \calA(q,r)\otimes \calA(p,q)\to \calA(p,r)
\end{gather*}
is induced by the composition of the cup product map for the cohomology of $\IA_n$ and the map $H^{q,r}\otimes H^{p,q}\cong \Hom(H^{\otimes r},H^{\otimes q})\otimes \Hom(H^{\otimes q},H^{\otimes p})\to H^{p,r}$ defined by the composition of linear maps.
The contraction map
\begin{gather*}
    \xi_j^i: \calA(p,q)\to \calA(p-1,q-1)
\end{gather*}
is induced by the contraction map $c_{i,j}: H^{p,q}\to H^{p-1,q-1}$.
We refer the reader to \cite[Definition 6.1, Proposition 6.2]{Kawazumi--Vespa}.

\begin{theorem}[Cf.\;{\cite[Conjecture 7.2]{HK}}]\label{TheoremAutIA}
    The inclusion map $i: \IA_n\hookrightarrow \Aut(F_n)$ induces an isomorphism of wheeled PROPs 
    \begin{gather*}
        i^*: \calH\xrightarrow{\cong} \calA,
    \end{gather*}
    and thus induces an isomorphism of $\Q[\gpS_p\times \gpS_q]$-modules
    \begin{gather*}
        i^*:H^*(\Aut(F_n),H^{p,q}) \to [H_A^*(\IA_n,\Q)\otimes H^{p,q}]^{\GL(n,\Z)}
    \end{gather*}
    for $n\ge \max(3*, 2\!*\!+p+q+3)$.
\end{theorem}

\begin{proof}
Since we have observed the stability of $H^*(\Aut(F_n),H^{p,q})$ and $[H_A^*(\IA_n,\Q)\otimes H^{p,q}]^{\GL(n,\Z)}$, by Theorem \ref{AutIA}, it suffices to check that the inclusion map $i: \IA_n\hookrightarrow \Aut(F_n)$ induces a surjective morphism of wheeled PROPs $i^*: \calH\twoheadrightarrow \calA$.

The Johnson homomorphism for $\Aut(F_n)$ is the group homomorphism
\begin{gather*}
    \tau: \IA_n\to \Hom(H,{\bigwedge}^2 H)
\end{gather*}
defined by $\tau(f)([x])=[f(x)x^{-1}]$ for $f\in \IA_n$ and $x\in F_n$, where $[x]$ denotes the image under the projection $F_n\twoheadrightarrow H$ and $[f(x)x^{-1}]$ the image under the projection $[F_n,F_n]\twoheadrightarrow {\bigwedge}^2 H$.
By abuse of the notation, let $\tau\in H^1(\IA_n,H^{2,1})$ denote the cohomology class that is represented by the composition map of the Johnson homomorphism and the inclusion $\Hom(H,{\bigwedge}^2 H)\hookrightarrow H^{2,1}$.
We can check that the cohomology class $\tau$ is non-trivial for $n\ge 3$ by using the dual pairing 
\begin{gather*}
    \langle,\rangle: H^1(\IA_n,H^{2,1})\otimes H_1(\IA_n,H^{1,2})\to \Q
\end{gather*}
defined by $\langle [f], [g\otimes x]\rangle =f(g)(x)$ for $f: \IA_n\to H^{2,1}\cong \Hom_{\Q}(H^{1,2},\Q)$ and $g\in \IA_n$, $x\in H^{1,2}$, where $[f]$ (resp. $[g\otimes x]$) denotes the cohomology class (resp. homology class).

Kawazumi \cite{Kawazumi} extended the Johnson homomorphism to a cohomology class $h_1\in H^1(\Aut(F_n),H^{2,1})$, which is non-trivial.
Therefore, we have
\begin{gather}\label{Kawazumicocycle}
    i^*(h_1)=\tau\in H^1(\IA_n,H^{2,1})^{\GL(n,\Z)}.
\end{gather}
By \cite{Kawazumi--Vespa} and \cite{Lindell}, $\calH$ is generated by $h_1$ as a wheeled PROP, which means that for any $p,q\ge 0$, the hom-space $\calH(p,q)$ is a $\Q[\gpS_p\times \gpS_q]$-module spanned by morphisms obtained from some copies of $h_1$ and $\id_H\in H^{1,1}$ by the horizontal composition, the vertical composition and the contraction map of $\calH$.
Since the map $i^*$ sends the generator $h_1$ of $\calH$ to $\tau\in \calA(2,1)$, in order to prove that $i^*$ is surjective, we have only to prove that $\calA$ is generated by $\tau$ as a wheeled PROP.

The Albanese cohomology of $\IA_n$ is generated by the first cohomology as an algebra.
Therefore, $\calA$ is generated by the degree $1$ part $[H^1_A(\IA_n,\Q)\otimes H^{p,q}]^{\GL(n,\Z)}=H^1(\IA_n,H^{p,q})^{\GL(n,\Z)}$.
For $p,q\ge 0$ with $p-q\neq 1$, we have $H^1(\IA_n,H^{p,q})^{\GL(n,\Z)}=0$ for $n\ge \max(3,p+q)$.
For $p=1,q=0$, we have
$$H^1(\IA_n,H^{1,0})^{\GL(n,\Z)}=\Q \xi^{1}_{1}(\tau)$$
for $n\ge 3$.
For $p=2,q=1$, we have
$$H^1(\IA_n,H^{2,1})^{\GL(n,\Z)}= \Q \{\tau, \xi^{1}_{1}(\tau)\cup \id_{H}, \id_{H}\cup \xi^{1}_{1}(\tau)\}$$ for $n\ge 3$.
For $p\ge 3$, $q=p-1$, the horizontal composition map
$$H^1(\IA_n,H^{2,1})^{\GL(n,\Z)}\otimes (H^{p-2,p-2})^{\GL(n,\Z)}\to H^1(\IA_n,H^{p,p-1})^{\GL(n,\Z)}$$
is surjective for $n\ge 2p-1$.
Therefore, $\calA$ is generated by $\tau$ as a wheeled PROP, which completes the proof.
\end{proof}

\end{document}